\documentclass[oneside]{amsart}

\usepackage[letterpaper,body={16.0cm,22.5cm}, mag=1000]{geometry}
\usepackage{amssymb}
\usepackage{amsthm}
\usepackage{amscd}
\usepackage{enumitem}
\usepackage{float}
\usepackage{placeins}
\usepackage{caption}
\usepackage{hyperref}

\usepackage{times}
\usepackage{graphicx}
\usepackage{tikz}
\usepackage{tikz-cd}
\usepackage[all,cmtip]{xy}

\numberwithin{equation}{section}
\theoremstyle{plain}

\newtheorem{cor}[equation]{Corollary}
\newtheorem{lemma}[equation]{Lemma}

\newtheorem{prop}[equation]{Proposition}

\newtheorem{thm}[equation]{Theorem}

\newtheorem{qns}[equation]{Question}

\newtheorem{conv}[equation]{Convention}

\theoremstyle{definition}
\newtheorem{defn}[equation]{Definition}

\newtheorem{exs}[equation]{Examples}

\newcommand{\dlabel}[1]{\ifmmode \text{\ttfamily \upshape [#1] } \else
{\ttfamily \upshape [#1] }\fi \label{#1}}

\newcommand{\Aut}{\operatorname{Aut} }

\newcommand{\Fix}{\operatorname{Fix} }

\begin{document}
\setcounter{page}{1}
\title[On indecomposable solutions to the YBE whose squaring map is a $p$-cycle]
{On indecomposable involutive solutions to the Yang-Baxter equation whose squaring map is a $p$-cycle}

\author{Marco Castelli}
\address{Department of Mathematics and Physics “Ennio De Giorgi”, University of Salento - Via per Arnesano, 73100 Lecce (LE), Italy}
\email{marco.castelli@unisalento.it}

\author{Arpan Kanrar}
\address{Department of Mathematical Sciences, Adamas University, Kolkata, India}
\email{arpan.kanrar.math@gmail.com}


\begin{abstract}
The pioneering work of Rump, which proved Gateva-Ivanova's conjecture concerning the decomposability of square-free solutions to the Yang-Baxter equation, significantly motivated further research into the associated squaring map $T$. This line of inquiry has yielded numerous decomposability theorems based on the underlying structure of $T$. Two seminal questions, posed by Ramírez and Vendramin, ask about the existence of certain indecomposable involutive solutions whose squaring maps are transpositions or $3$-cycles. In this paper, we explore these problems by examining the case where $T$ is a $p$-cycle, for an arbitrary prime number $p$. We provide negative answers to the aforementioned questions under the assumption that the solution has nilpotent permutation group or has prime-power cardinality, providing also some decomposition criteria in this setting. Moreover, we show that, in the particular case of latin solutions, the situation is more rigid. 
 \end{abstract}
\maketitle

\section*{Introduction}
The quantum Yang–Baxter equation first emerged in theoretical physics, in a paper by C.N. Yang \cite{Yang}, and in statistical mechanics, in the work of R.J. Baxter \cite{Baxter}. Later, it was also investigated from a mathematical perspective. In 1992, Drinfel'd \cite{Drinfeld92} proposed the study of \emph{set-theoretical} solutions, namely, solutions that act on a basis of the underlying vector space.  Specifically, a \emph{set-theoretic solution of the Yang-Baxter equation} on a non-empty set $X$ is a pair $\left(X,r\right)$, where 
$r:X\times X\to X\times X$ is a map such that the relation
$$(r\times id_X)(id_X \times r)(r\times id_X)=(id_X \times r)(r\times id_X)(id_X\times r) $$
is satisfied. Set-theoretic solutions yield solutions to the Yang–Baxter equation via linearization. Writing a solution $(X,r)$ as $r\left(x,y\right) = \left(\lambda_x\left(y\right)\rho_y\left(x\right)\right)$, with
$\lambda_x, \rho_x$ maps from $X$ into itself, for every $x\in X$, we say that $(X, r)$ is \emph{non-degenerate} if $\lambda_x,\rho_x\in Sym_X$, for every $x\in X$, and \emph{involutive} if $r^2=id_{X\times X}$.\\
The seminal papers by Gateva-Ivanova and Van Den Bergh  \cite{GI-V} and Etingov, Schedler, and Soloviev \cite{ESS99} inspired many researchers to study involutive, non-degenerate set-theoretic solutions (which we simply refer to as \emph{involutive solutions}). To understand and analyze an involutive solution, a successful approach is based on the notion of decomposable solution. Recall that an involutive solution $(X,r)$ is said to be \emph{decomposable} if $X$ can be expressed as a disjoint union of two non-empty subsets $Y_1$ and $Y_2$ (that form a \emph{decomposition} $\{Y_1,Y_2\}$ of $X$) such that $r(Y_i\times Y_i)\subseteq Y_i\times Y_i$ for all $i\in \{1,2\}$, and $(X,r)$ is called \emph{indecomposable}, otherwise. In some sense, indecomposable solutions are the fundamental building blocks of this theory (see \cite[Section $3$]{ESS99}).\\
Indecomposability of solutions has been extensively studied in the last few years, and several results have been obtained by various techniques. According to \cite[Proposition 2.12]{ESS99} an involutive solution is indecomposable if and only if the \emph{associated permutation group}, i.e. the subgroup of $Sym(X)$ generated by the maps $\lambda_x$, acts transitively on $X$. This fact is of crucial importance for several papers, in which indecomposable solutions whose associated permutation group belongs to a specific class have been extensively studied (see for example \cite{CJO,jedlivcka2021cocyclic,JP, CommAlg}). In \cite[Theorem 7]{cacsp2018}, the indecomposability of solutions obtained by dynamical extensions was completely characterized, while in \cite{CA23IND} and in \cite{ArpanArch} some sufficient conditions, involving the permutation group $\mathcal{G}(X)$, that guarantee the decomposability of an involutive solution were exhibited.
Several mathematicians have focused on the so-called \emph{squaring map} $T$, which is the permutation of $Sym(X)$ given by $T(x):=\lambda_x^{-1}(x)$, for all $x\in X$. The success of this approach is due to the fact that this single permutation contains information about the decomposability of a solution. In this direction, the first paper was authored by Rump \cite{Rump05}, who showed that if the squaring map is the identity map, then the solution must be decomposable. In 2021, Ramírez and Vendramin provided several (in)decomposability theorems based on the cycle structure of $T$ \cite{VenRami22}. In particular, they showed the following results:
\begin{itemize}
    \item[1)] if $T$ is an $n$-cycle, where $n=|X|$, then $(X,r)$ is indecomposable;
    \item[2)] if $T$ is an $(n-1)$-cycle, where $n=|X|$, then $(X,r)$ is decomposable;
    \item[3)] if $T$ is an $(n-2)$-cycle, where $n=|X|$ and is odd, then $(X,r)$ is decomposable;
    \item[4)] if $T$ is an $(n-3)$-cycle, where $n=|X|$ and is coprime with $3$, then $(X,r)$ is decomposable.
\end{itemize}
A striking result, due to C. Mora and Sastriques \cite{CS23}, provides an extension of the previous decomposability criteria and of Rump's result. Indeed, in \cite[Theorem A]{CS23} it was shown that if  the size of $X$ and the order of the map $T$ are coprime, then the solution must be decomposable. Further generalizations to non-involutive solutions are contained in \cite{colazzo2024cabling}. However, the map $T$ does not seem to be merely a tool for identifying the decomposability of a solution. In this regard, in a recent paper \cite{DI25END}  Dietzel proved that involutive solutions of order $p^k$ whose squaring maps are full cycles (and therefore these solutions are indecomposable) belong to the class of involutive solutions of finite \emph{multipermutation level} if $p>2$, and he showed that involutive solutions of order $2^k$ that are not of finite multipermutation level have an iterated-retraction of size $4$. Despite these interesting results, the relation between the structure of an indecomposable solution and its squaring map remains poorly understood. To partially fill this gap, in \cite{VenRami22} Ramírez and Vendramin  computed indecomposable solutions of  size $\leq 10$ and the cycle decompositions of their squaring maps. As noted in \cite[Section $3$]{VenRami22}, among small involutive solutions, there is only one indecomposable solution with a 3-cycle as squaring map and two indecomposable solutions with a transposition as squaring map. Motivated by these pieces evidence, they asked if there are further indecomposable involutive solutions of this type (see Questions 3.18 and 3.19 of \cite{VenRami22}). Moreover, they pointed out that, among the solutions of small order, the squaring maps of the indecomposable ones have ``few" fixed points (see Question 3.17 of \cite{VenRami22} and the comment below).\\
Motivated by this background, in this paper we explore the class of indecomposable solutions whose squaring map is a $p$-cycle, where $p$ is an arbitrary prime number. Our strategy can be summarized as follows. First, we develop a theory of indecomposable solutions with a nilpotent permutation group. As a first application, we provide a decomposability criterion that allow us to detect decomposable solutions that are different from the ones considered in \cite{CS23,CA23IND,ArpanArch}. Then, using an idea implicitly suggested in \cite{CJO}, we make use of the information provided by the imprimitivity blocks of $\mathcal{G}(X)$ to provide our structural results on indecomposable solutions with a $p$-cycle as squaring map.  In this way, we give a negative answer to Questions 3.18 and 3.19 of \cite{VenRami22} under the assumption that the associated permutation group is nilpotent or $X$ has a prime-power cardinality. These two hypotheses seem reasonable, since, by a well-known result due to Etingof, Schedler, and Soloviev, the associated permutation group is always solvable \cite[Theorem 2.15]{ESS99}. In this context, among other results, we will use our results on blocks of imprimitivity to answer positively a question posed by Agata and Alicja Smoktunowicz in \cite{smock} concerning the decomposability of certain cycle sets, assuming that the associated permutation group is nilpotent. Moreover, by the so-called \emph{cabling} technique, if we assume that X has prime-power cardinality, we can drop the nilpotency assumption on the permutation group.  At the end, using different techniques, we examine the case of latin solutions (a class of indecomposable solutions that has recently been studied \cite{BOC25,BKSV}), showing that the situation is even more rigid in this setting. 

\section{Basics on left braces and cycle sets}

In this section, we provide basic definitions and results on left braces and cycle sets that are useful throughout the paper.
\smallskip

Following \cite[Definition 1]{CO14}, a set $B$ endowed with two operations $+$ and $\circ$ is said to be a \textit{left brace} if $(B,+)$ is an abelian group, $(B,\circ)$ is a group, and
$$
    x\circ (y + z) + x
    = x\circ y + x\circ z,
$$
for all $x,y,z\in B$. The group $(B,+)$ will be called the \emph{additive} group of the left brace, while the group $(B,\circ)$ will be called the \emph{multiplicative group.} We will indicate by $-a$ the inverse of an element $a \in B$ with respect to the operation $+$, and by $a^-$ the inverse of an element $a \in B$ with respect to the operation $\circ$. Moreover, the order of an element $b$ of $B$ in the additive group $(B,+)$ (resp. multiplicative group $(B,\circ)$) will be denoted by $o(b)_+$ (resp. $o(b)$).

\begin{exs}\label{esbrace}
\begin{itemize}
    \item[1)] If $(B,+)$ is an abelian group, then the operation $\circ$ given by $x\circ y:=x+y$ give rise to a left brace which we will call \emph{trivial}.
\item[2)] Let $B:=(\mathbb{Z}/p^2\mathbb{Z},+)$ and $\circ$ be the binary operation on $B$ given by $x\circ y:=x+y+p\cdotp x\cdotp y$ (where $\cdotp$ is the ring-multiplication of $\mathbb{Z}/p^2\mathbb{Z}$). Then, $(B,+,\circ)$ is a left brace.\end{itemize}

\end{exs}


\smallskip

\noindent Given a left brace $B$ and $x\in B$, let us denote by $\lambda_x:B\longrightarrow B$ the map from $B$ into itself defined by
\begin{equation}\label{eq:gamma}
    \lambda_x(y):= - x + x\circ y,
\end{equation} 
for all $y\in B$. Let us recall the properties of the maps $\lambda_x$.
\begin{prop}\label{action}\cite[Proposition 2]{Rump07},\cite[Lemma 1]{CO14}
Let $B$ be a left brace. Then, the following are satisfied: 
\begin{itemize}
\item[1)] $\lambda_x\in\Aut(B,+)$, for every $x\in B$;
\item[2)] the map $\lambda:B\longrightarrow \Aut(B,+)$, $x\mapsto \lambda_x$ is a group homomorphism from $(B,\circ)$ into $\Aut(B,+)$.
\end{itemize}

\end{prop}

\noindent For the following definition, we refer the reader to \cite[pg. 160]{Rump07} and \cite[Definition 3]{CO14}.

\begin{defn}
Let $B$ be a left brace. A subset $I$ of $B$ is said to be a \textit{left ideal} if it is a subgroup of the multiplicative group and $\lambda_x(I)\subseteq I$, for every $x\in B$. Moreover, a left ideal is an \textit{ideal} if it is a normal subgroup of the multiplicative group.
\end{defn}

\noindent 
\noindent A standard ideal of a left brace $B$ is the \emph{socle}, indicated by $Soc(B)$ and given by the kernel of the map $\lambda$. 

\medskip

An algebraic structure $(X,\cdot)$ is said to be a \emph{cycle set} if each left multiplication $\sigma_x:X\longrightarrow X$, $y\mapsto x\cdotp y$ is bijective and the following  equality (called the \emph{cycloid equation})
\begin{equation}\label{eqcycleset}
    (x\cdot y)\cdot (x\cdot z)=(y\cdot x)\cdot (y\cdot z),
\end{equation}
holds for all $x,y,z\in X$.  Moreover, a cycle set $(X,\cdot)$ is called \textit{non-degenerate} if the squaring map $T:X\longrightarrow X$, $x\mapsto x\cdot x$ is bijective. 

\begin{exs}\text{ }
\begin{itemize}
 \item[1)] If $X$ is a nonempty set and $\gamma \in Sym(X)$, the binary operation given by $x\cdotp y:=\gamma(y)$ makes $X$ into a non-degenerate cycle set. We will refer to these cycle sets as the \emph{trivial} cycle set induced by the permutation $\gamma$. 
 \item[2)] If $(X,\cdotp),(Y,\bullet)$ are cycle sets, then the algebraic structure $(X\times Y,*)$ given by $(x,y)*(z,t):=(x\cdotp z,y\bullet t)$ is a cycle set which we will call the \emph{direct product} of $X$ and $Y$.
\end{itemize}
\end{exs}

\noindent An useful tool to construct new cycle sets from a given one, introduced in \cite{ESS99}, is the so-called \textit{retract relation}.
Specifically, in \cite{Rump05} Rump showed that the binary relation $\sim_\sigma$ on $X$ given by 
$$x\sim_\sigma y :\Longleftrightarrow \sigma_x = \sigma_y$$ 
for all $x,y\in X$, is a \emph{congruence} of $(X,\cdot)$, i.e. an equivalence relation for which $x\sim_\sigma y$ and $x'\sim_\sigma y'$ imply $x\cdotp x'\sim_\sigma y\cdotp y',$ for all $x,x',y,y'\in X$. In \cite{ESS99} (and independently in \cite{Rump05}) it was shown that the quotient $X/\sim_{\sigma}$, which we denote by $Ret(X)$, is a cycle set, which we will call \emph{retraction} of $(X,\cdot)$. 
A cycle set $X$ is said to be \textit{irretractable} if $Ret(X)=X$, otherwise it is called \textit{retractable}.


\noindent In a classical manner we can define the notion of cycle sets homomorphism.

\begin{defn}
    Let $X,Y$ be cycle sets. A map $p:X\longrightarrow Y$ is said to be a \emph{homomorphism} between $X$ and $Y$ if $p(x\cdotp y)=p(x)\cdotp p(y)$ for all $x,y\in X$. A surjective homomorphism is called \emph{epimorphism}, while a bijective homomorphism is said to be an \emph{isomorphism}.
\end{defn}

\noindent Two standard permutation groups related to a cycle sets $X$ are the one generated by the set $\{ \sigma_x|\hspace{1mm} x\in X\}$, called the \emph{associated permutation group of $X$} and indicated by $\mathcal{G}(X)$, and the one generated by the set $\{ \sigma_x\sigma_y^{-1}|\hspace{1mm} x,y\in X\}$, called the \emph{displacements group of $X$} and indicated by $Dis(X)$. In this context, our attention will be focused on indecomposable cycle sets.

\begin{defn}
A cycle set $(X,\cdot)$ is said to be \textit{indecomposable} if the permutation group $\mathcal{G}(X)$ acts transitively on $X$. 
\end{defn}

\noindent Epimorphisms between cycle sets have a crucial role in the study of the indecomposable ones, as highlighted in \cite{cacsp2018,COK}. In that regard, the following lemma is of crucial importance for our purposes.

\begin{lemma}[Lemma 1 of \cite{cacsp2018}]\label{dyn}
    Let $X$ be a finite indecomposable cycle set, $Y$ a cycle set and $p:X\longrightarrow Y $ a cycle set epimorphism. Then, the fibers of $p$, i.e. the sets given by $p^{-1}(y):=\{x\in X\mid p(x)=y \}$, have all the same cardinality.
\end{lemma}

\noindent Cycle sets are useful to construct new solutions of the Yang-Baxter equation.

\begin{prop}[Propositions 1-2, \cite{Rump05}]\label{corrisp}
Let $(X,\cdot)$ be a cycle set. Then the pair $(X,r)$, where $r(x,y):=(\sigma_x^{-1}(y),\sigma_x^{-1}(y)\cdot x)$, for all $x,y\in X$, is an involutive left non-degenerate solution of the Yang-Baxter equation which we call the associated solution to $(X,\cdot)$. Moreover, this correspondence is one-to-one.
\end{prop}

\noindent Of course, under the previous correspondence, indecomposable cycle sets correspond to indecomposable solutions.

\begin{conv}
    From now on, every cycle set will be assumed to be non-degenerate. All the results will be given in the language of cycle sets, but they can be translated by Proposition \ref{corrisp}.
\end{conv}

Let $B$ be a left brace. A \emph{cycle base} is a subset $X$ of $B$ that is a union of orbits with respect to the $\lambda$-action and that generates the additive group $(B,+)$. If a cycle base is a single orbit, then is called \emph{transitive cycle base}. By \cite[Theorem 3]{RumpForum}, every indecomposable cycle sets can be constructed by transitive cycle bases of left braces; on the other hand, every cycle set $X$ gives rise to a left brace structure on $\mathcal{G}(X)$ having as $\circ$ operation the usual maps composition and with sum induced by the rule $\sigma_x^{-1}+\sigma_y^{-1}:=\sigma_x^{-1}\circ \sigma_{\sigma_x(y)}^{-1}$ for all $x,y\in X$ (see Section $1$ of \cite{RumpForum}). 
\begin{prop}[Theorem 3, \cite{RumpForum}]\label{cosmod}
Let $(B,+,\circ)$ be a left brace, $Y\subset B$ a transitive cycle base, $a\in Y$,
and $K$ a core-free subgroup of $(B,\circ)$, contained in the
stabilizer  $B_{a}$ of $a$ (respect to the action $\lambda$). Then, the pair $(X,\cdotp)$ given by $X:=B/K$ and $\sigma_{x\circ K}(y\circ K):=\lambda_x(a)^-\circ y \circ K$ gives rise to an indecomposable cycle set with $ \mathcal{G}(X)\cong B$.\\
 Conversely, every indecomposable cycle set $(X,\cdotp)$ can be obtained in this way by taking $B:=\mathcal{G}(X)$.
\end{prop}

\noindent A different tool that allows one to construct indecomposable cycle sets starting from a given one is the so-called cabling. 

\begin{defn}
    Let $(X,\cdotp)$ be a cycle set. Define inductively a family of binary operations $\cdotp_k$ on $X$ by $x \cdotp_1 y:=x \cdotp y$ and $x\cdotp_k y:= (x\cdotp_{k-1} x)\cdotp (x\cdotp_{k-1} y)$ for all $k\in \mathbb{N}$ and $x,y\in X$. Then, the pair $(X,\cdotp_{k})$ is called the $k$-\emph{cabled} of $(X\cdotp)$.
\end{defn}

\begin{prop}[Sections $2$ and $3$, \cite{LRV24}]\label{cablind}
    Let $(X,\cdotp)$ be a cycle set and $k\in \mathbb{N}$. Then:
    \begin{itemize}
        \item $(X,\cdotp_k)$ is itself a cycle set;
        \item if $(X,\cdotp)$ is indecomposable and $(k,|X|)=1$, then $(X,\cdotp_k)$ also is indecomposable;
        \item the squaring map of $(X,\cdotp_k)$ coincides with $T^k$;
        \item the left brace $\mathcal{G}(X,\cdotp_k)$  is equal to $k\mathcal{G}(X)=\{kb  \mid  b\in \mathcal{G}(X) \}$.
    \end{itemize}
\end{prop}

\noindent By \cite[Section 3]{Okcycle}, we know that if $X$ is an indecomposable cycle set and $I$ is an ideal of the left brace $\mathcal{G}(X)$, the $I$-orbits on $X$ give rise to a canonical cycle set structure $X/I$ such that the canonical map $X\mapsto X/I$ is an epimorphism. On the other hand, a cycle set $X$ is said to be \emph{simple} if for every epimorphism $p:X\longrightarrow Y$ we have that $X\cong Y$ or $|Y|=1$. The following proposition provides structural results on simple cycle sets. 

\begin{prop}(\cite[Propositions 4.1 and 4.2]{COK})\label{cok}
    Let $X$ be a simple cycle set. If $|X|>2$, then $X$ is indecomposable, and if $|X|$ is not a prime number, then $X$ is irretractable.
\end{prop}


\noindent We close this introductory section by recalling another important invariant of a cycle set. If $X$ is a cycle set, the map $\Omega_d$ from $X^d$ to $X$ (where $X^d$ is the Cartesian product of $X$ taken $d$ times) is inductively given by  $$\Omega_1(x_1)=x_1\hspace{2mm}  \mbox{and} \hspace{2mm}\Omega_d(x_1,...,x_d):=\Omega_{d-1}(x_1,...,x_{d-1})\cdotp \Omega_{d-1}(x_1,...,x_{d-2},x_d)$$
    for all $d\in \mathbb{N}$, $d>1$ and $x_1,...,x_d\in X$.

\begin{defn}
    Let $(X,\cdotp)$ be an involutive solution. Then, $(X,\cdotp)$ has \emph{Dehornoy's class $d$} if $d$ is the minimal positive integer such that $\Omega_{d+1}(x,x,...,x,y)=y$ for all $x,y\in X$.
\end{defn}

\noindent From now on, if $n$ is a natural number, we denote the set of primes dividing $n$ by $\pi(n)$

\begin{prop}[Section 2 of \cite{fe24de} and Theorem $G$ of  \cite{LRV24}]\label{propdeho}
    Let $X$ be a finite indecomposable cycle set and $d$ its Dehornoy's class. Then:
    \begin{itemize}
        \item $\pi(d)=\pi(|\mathcal{G}(X)|)$;
        \item $d$ is equal to the exponent of the additive group of the left brace on $\mathcal{G}(X)$;
        \item $o(T)$ divides $d$;
        \item $\pi(o(\sigma^{-1}_x)_+)=d$ for all $x\in X$.
    \end{itemize}
\end{prop}

\section{Indecomposable cycle sets with nilpotent permutation group}

In this section, we consider indecomposable cycle sets whose permutation group $\mathcal{G}(X)$ is a direct product of groups having relatively coprime orders. We will show that the structure of these cycle sets is influenced by the structure of the associated left braces.

\smallskip

We start with a lemma that provides a slight extension of \cite[Theorem 19]{smoktunowicz18}.

\begin{lemma}\label{brace}
    Let $(B,+,\circ)$ be a left brace such that $(B,\circ)$ is the direct product of subgroups $B_i$ whose orders are relatively coprime. Then, each $B_i$ is an ideal, and $B$ is a left brace isomorphic to the direct product of $B_i$.
\end{lemma}

    \begin{proof}
        We can assume $|B|=m_1m_2\cdots m_r$, where $|B_i|=m_i$. Let $A_i$ be the additive subgroup of order $m_i$, then $A_i$ is a characteristic subgroup, so a left ideal. Since by \cite[Theorem $2.15$]{ESS99} $(B,\circ)$ is solvable, $A_i$ is a conjugate of $B_i$, and the normality of $B_i$ gives $B_i=A_i$. We obtain that each $B_i$ is an ideal. Now, if $i,j\in \{1,...,r\}$ with $i\neq j$, and $x,y\in B$ with $x\in B_i$, $y\in B_j$, we have $x\circ y=x+\lambda_x(y)=\lambda_x(y)+x=\lambda_{x}(y)\circ \lambda_{\lambda_{x}(y)^-}(x)=\lambda_{\lambda_{x}(y)^-}(x)\circ \lambda_x(y)$ and hence $\lambda_x(y)=y$. This gives $x_1\circ ... \circ x_r=x_1+...+x_r$ for all $x_1\in B_1,$..., $x_r\in B_r$, and therefore the statement is proven. 
        \end{proof}


\noindent Below, we give the main result of this section, that extends Theorem $10$ of \cite{CaCaSt20x}.

\begin{thm}\label{nilp}
    Let $X$ be a finite indecomposable cycle set such that the the permutation group is a direct product of $r$ subgroups $B_i$ whose orders are relatively coprime. Then $X$ is isomorphic to the direct product of $r$ cycle sets $X_i$ such that $|X_i|>1$ and $|X_i|$ divides $|B_i|$.
\end{thm}

\begin{proof}
   By Lemma \ref{brace} the left brace on $B:=\mathcal{G}(X)$ is isomorphic to the direct product of the ideals $B_i$. By Proposition \ref{cosmod}, there exists $a\in B$ such that $X$ can be constructed as a cycle set on $B/K$, where $K$ is a core-free subgroup of $(B,\circ)$ contained in the stabilizer $B_{a}$ of an element $a\in B$. Since the $\lambda$ maps, the inverse operator and the $\circ$ operation are defined component-wise on the ideals $B_i$, and the subgroup $K$ is the direct product of subgroups $K_i$, with $K_i$ a core-free subgroup of $B_i$, we obtain that $X$ is isomorphic to the direct product of $r$ cycle sets $X_i$. Since $K$ is core-free, we cannot have $K_i=B_i$ for some $i$, hence $|X_i|>1$ and $|X_i|$ divides $|B_i|$.
\end{proof}

\noindent Of course, the previous theorem applies when $\mathcal{G}(X)$ is nilpotent and the $B_i$ are the Sylow subgroups. In \cite{ArpanArch} the second author showed that in this case $\pi(|X|)=\pi(|\mathcal{G}(X)|)$. In this context, the following results provide additional information and a further decomposition criteria.

\begin{prop}\label{pitype}
    Let $X$ be a finite indecomposable cycle set with nilpotent permutation group and $x\in X$. Then, $\pi(o(\sigma_x))=\pi(|X|)=\pi(|\mathcal{G}(X)|)=\pi(o(T))$.
\end{prop}
\begin{proof}
    Since $\mathcal{G}(X)$ is nilpotent, the second equality follows by Theorem \ref{nilp} (or by \cite[Corollary 3.10]{ArpanArch}). Now, if $|X|=p^n$ for some prime number $p$, then $\sigma_x$ cannot be the identity by \cite[Theorem 3.2]{VenRami22}, and hence it must have order $p^k$ for some $k\in \mathbb{N}$, moreover $T$ can not be the identity by the main theorem of \cite{CS23}, and hence it also must have prime-power order by Proposition \ref{propdeho}.  If $\pi(|X|)>1$, the statement follows by Theorem \ref{nilp} and an easy inductive argument. 
\end{proof}

\begin{cor}\label{prdec}
        Let $X$ be a finite cycle set with nilpotent permutation group and suppose that one of the following conditions hold:
        \begin{itemize}
            \item there is an element $x\in X$ such that $\pi(o(\sigma_x))\neq\pi(|X|)$;
            \item there is an element $x\in X$ such that $\pi(o(\sigma_x))\neq\pi(|\mathcal{G}(X)|)$;
            \item $\pi(o(T))\neq \pi(|X|)$;
            \item $\pi(o(T))\neq \pi(|\mathcal{G}(X)|)$.
        \end{itemize}  
        Then, $X$ is decomposable.
\end{cor}

\noindent In \cite{ArpanArch}, the second author showed that if a cycle set $X$ has nilpotent permutation group and $\pi(|X|)\neq \pi(|\mathcal{G}(X)|)$, then $X$ is decomposable. Proposition \ref{pitype} and the previous corollary  allow us to detect a larger class of decomposable cycle sets, as the following example shows.

\begin{exs}
    Let $(X,\cdotp)$ be the cycle set given by $X:=\{1,...,12\}$, $\sigma_1=\sigma_2:=(1,2)$, and $\sigma_x:=(3,4,5)$ for all $x\in\{3,...,12\}$. Then, $\mathcal{G}(X)$ is the cyclic group of size $6$, $\pi(|X|)=\pi(|\mathcal{G}(X)|)=\{2,3\}$, but the solution is decomposable since $\pi(o(\sigma_1))\neq \{2,3\}$. Moreover, the decomposability of $X$ does not follow from \cite[Theorem $A$]{CS23} because $(o(T),|X|)=6$.
\end{exs}

\section{Indecomposable cycle sets with a p-cycle as squaring map}

In this section, we study indecomposable cycle sets whose squaring map is a $p$-cycle, with an emphasis on the ones having prime-power size. In the first part of this section, we present some general results, culminating in the proof of the main theorems.

\begin{prop}\label{prop:1}
    Let $X$ be an indecomposable finite cycle set such that the squaring map is a $p$-cycle, for a prime number $p$. Then, $X$ is simple. In particular, if $|X|$ is not a prime number, then $X$ is irretractable and $Soc(\mathcal{G}(X))=\{0\}$.
\end{prop}
\begin{proof}
    Suppose that $X$ is not simple. Let $Y$ be a non-trivial epimorphic image of $X$ and $f$ be an epimorphism from $X$ to $Y$. Then, the order of the squaring map $T_Y$ of $Y$ divides that of the squaring map $T_X$ of $X$, hence $o(T_Y)$ is $1$ or $p$.
    In the first case, $Y$ is square-free and by \cite[Theorem 1]{Rump05} decomposable, a contradiction. In the second case, since by Lemma \ref{dyn}  the fibers of $f$ are of the same size, we conclude that the squaring map of $X$ cannot be a $p$-cycle since it moves at least $2p$ elements, again a contradiction. Therefore, $X$ is simple. By Proposition \ref{cok} if $|X|$ is not a prime number then $X$ is irretractable, and by \cite{BA15FA} $Soc(\mathcal{G}(X))=\{0\}$.
\end{proof}

\noindent A cycle set is said to be of $\pi$-type if  $\pi(|X|)=\pi(|\mathcal{G}(X)|)$.

\begin{prop}\label{cablpi}
    Let $X$ be a finite indecomposable cycle set. Then, there exists a cabled $(X,\cdotp_l)$ of $X$ that is an indecomposable cycle set of $\pi$-type. Moreover, if $T$ is a $p$-cycle for a prime number $p$, then the squaring map of $(X,\cdotp_l)$ is a $p$-cycle. 
\end{prop}
\begin{proof}
   Suppose that $|\mathcal{G}(X)|=l\cdotp t$ for some number $l$ such that $(l,|X|)=(l,t)=1$, and such that $\pi(t)=\pi(|X|)$. Let $(X,\cdotp_{l})$ be the $l$-cabled of $X$. By Proposition \ref{cablind}, $(X,\cdotp_{l})$ is indecomposable and $\mathcal{G}(X,\cdotp_l)=l\mathcal{G}(X)$, hence $|\mathcal{G}(X,\cdotp_l)|=t$ and the first part of the statement follows. The last part follows again by Proposition \ref{cablind} and the fact that if $T$ is a $p$-cycle with $p$ prime, then $T^k$ is again a $p$-cycle for all natural numbers $k$ that are coprime to $p$. 
\end{proof}

\begin{cor}\label{disnil}
    Let $X$ be an indecomposable finite cycle set such that the squaring map is a $p$-cycle, for a prime number $p$. Moreover, suppose that $Dis(X)$ is nilpotent. Then, $|X|=p^k$ for some natural number $k$.
\end{cor}

\begin{proof}
    By Proposition \ref{prop:1} $X$ is simple and by \cite[Theorem A]{CS23} $X$ is a multiple of $p$. If $|X|=p$, the statement clearly follows. If $|X|\neq p$, by \cite[Corollary 41]{BOC25} we have that $X$ must have prime-power size, and the statement again follows. 
\end{proof}

\noindent The following lemma, which is of crucial importance for our purposes, gives a nice combinatorial property of a (not necessarily indecomposable) cycle set.

\begin{lemma}[Lemma 2.12, \cite{BKSV}]\label{mapbij}
    Let $X$ be a non-degenerate cycle set. Then, the map $M:X\times X\longrightarrow X\times X$, $(x,y)\mapsto (x\cdotp y,y\cdotp x)$ is bijective.
\end{lemma}

The rest of this section is devoted to the main structural results. From now on, we denote by $\Fix(T)$ the set of the fixed point under the squaring map.

\begin{lemma}\label{sysblocks}
    Let $X$ be an indecomposable finite cycle set with nilpotent permutation group, and let $p$ be the minimal prime that divides $|\mathcal{G}(X)|$. Then $X$ has a system of blocks $I:=\{\Delta_x\}_{x\in X}$ of size $p$ and there exists $z\in X$ such that the image of $\sigma_z$ in $Sym(I)$ is a $p$-cycle and if $x\in \Fix (T)$ then $\sigma_x$ in $Sym(I)$ is trivial.
\end{lemma}

\begin{proof}
    Since $X$ is of $\pi$ type, there exist two normal subgroups $H$ and $K$ of $\mathcal{G}(X)$ such that $|H/K|=p$, $H$ acts transitively on $X$ and $K$ acts intransitively on $X$. Then, the orbits of $K$ induce a system of blocks $I$, and since $|H/K|=p$, we have that $|I|=p$. Now, consider the action of $\mathcal{G}(X)$ on $I$. Let $\sigma\in\mathcal{G}(X),$ the minimality condition implies $\sigma(\Delta)=\Delta$ for all $\Delta\in I$, or it acts as a $p$-cycle on the elements of $I$. Therefore, if $x\in \Fix(T)$, then $\sigma_x$ in $Sym(I)$ is trivial. Moreover, the image of $\mathcal{G}(X)$ in $Sym(I)$, which we denote by $\mathcal{G}(X)_I$, is a group that acts transitively on $I$, hence the existence of $z\in X$ such that the image of $\sigma_z$ in $Sym(I)$ is a $p$-cycle follows from the fact that $\{\sigma_x~|~x\in X\}$ generates $\mathcal{G}(X)$.
\end{proof}

\noindent As an elementary consequence of the previous lemma, we obtain the following corollary.

\begin{cor}
    Let $X$ be an indecomposable finite cycle set with nilpotent permutation group. Then, there exists $z\in X$ such that $\sigma_z$ has no fixed points.
\end{cor}

In \cite[Question 6.7]{smock}, Agata and Alicja Smoktunowicz asked whether the following statement holds: if $X$ is a cycle set such that for every $x\in X$ there exists $y\in X$ satisfying $ x\cdotp y=y \quad\text{and}\quad y\cdotp x=x, $
then $X$ is decomposable. They proved that this holds whenever $X$ has finite multipermutation level. As an easy consequence of the previous corollary, we obtain that this holds if $X$ has nilpotent permutation group, and moreover we have that the first of these conditions is sufficient.

Now, we can give the desired classification results.

\begin{thm}\label{classc2}
    Let $X$ be an indecomposable cycle set with nilpotent permutation group and with $T=(1\quad2)$. Then, $|X|\in \{2,4\}$. Moreover, for each size, the cycle set is unique.
\end{thm}
\begin{proof}
    By Proposition \ref{pitype}, we have that $|X|=2^s$ for some $s\in \mathbb{N}$ and $\mathcal{G}(X)$ is a $2$-group. Now, suppose that $|X|>4$. By Lemma \ref{sysblocks}, we have a system of blocks of size $2$ and since $T=(12)$, necessarily it follows that $ \Delta_1\neq \Delta_2$. Since $\sigma_{x\cdotp y}\sigma_x=\sigma_{y\cdotp x}\sigma_y$ for all $x,y\in X$, if $x,y\in \Fix T$ we must have $M(x,y)=(y\cdotp x,x\cdotp y)\in (\Fix T\times \Fix T)\cup ((X-\Fix T)\times (X-\Fix T)) $ (otherwise, $\sigma_{x\cdotp y}\sigma_x$ and $\sigma_{y\cdotp x}\sigma_y$ acts differently on the system of blocks). In the same way, if $x,y\in X-\Fix T$, we must have $M(x,y)=(y\cdotp x,x\cdotp y)\in (\Fix T\times \Fix T)\cup ((X-\Fix T)\times (X-\Fix T)) $. If $M(x,y)\in (X-\Fix T)\times (X-\Fix T)$ for all $x,y\in (X-\Fix T)\times (X-\Fix T)$, since $M$ is bijective, we must have $M(x,y)\in \Fix T\times \Fix T$ for all $x,y\in \Fix T$. This implies that $\Fix T$ and $X-\Fix T$ are sub-cycle set of $X$ and hence $X$ is decomposable, a contradiction. \\
    Therefore, there exist $a,b\in \Fix T$ such that $M(1,2)=(2\cdotp 1,1\cdotp 2)=(a,b)$. Moreover, $a\in \Delta_2$ and $b\in \Delta_1$.\\ 
    Now, 
    \begin{align*}
        a\cdot 1 &=(2\cdot 1)\cdot (2\cdot 2)=(1\cdot 2)\cdot (1\cdot 2)=b\cdot b=b,
    \end{align*}
    \begin{align*}
        b\cdot 2 &=(1\cdot 2)\cdot (1\cdot 1)=(2\cdot 1)\cdot (2\cdot 1)=a\cdot a=a,
    \end{align*}
    and
   we have shown that
    \begin{equation}\label{eq:M}
        M\Big((\Fix T\times \Fix T)\cup\big( (X-\Fix T)\times (X-\Fix T) \big)\Big)=(\Fix T\times \Fix T)\cup\big( (X-\Fix T)\times (X-\Fix T) \big),\\
    \end{equation}
therefore:
    $$M(1,a)=(b,1\cdot a)\in (\Fix T\times (X-\Fix T))\mbox{ and } M(2,b)=(a,2\cdot b)\in (\Fix T\times (X-\Fix T)).$$
Also we have $a\in \Delta_2$ and $b\in \Delta_1$, thus $1\cdot a=1$ and $2\cdot b=2.$\\
    Now, $a\in \Fix T\cap \Delta_2$ and $b\in \Fix T\cap \Delta_1$, therefore $2\cdot a,~ 1\cdot b\notin \{1,2\}$. Hence by \eqref{eq:M}, $M(2,a)=(a\cdot 2, 2\cdot a),~M(1,b)=(b\cdot 1, 1\cdot b)\in \{1,2\}\times \Fix T$ and moreover $\sigma_a$ and $\sigma_b$ fixes $\Delta_1$ and $\Delta_2$, so it follows that $a\cdot 2=2$ and $b\cdot 1=1$.
Hence we have the following parzial multiplication table
\[
\begin{array}{c|cccc}
 & 1 & 2 & a & b \\
\hline
1 & 2 & b & 1 & - \\
2 & a & 1 & - & 2 \\
a & b & 2 & a & - \\
b & 1 & a & - & b \\
\end{array}
\]
Note that, actually, we do not know if $\{1,2,a,b\}$ is closed under $\cdotp$. Now, since $2\cdotp b=2$ and $1\cdotp a=1$ we have
$$1=(2\cdotp 2)=(2\cdotp b)\cdotp (2\cdotp b)=(b\cdotp 2)\cdotp (b\cdotp b)=a\cdotp b$$
$$2=(1\cdotp 1)=(1\cdotp a)\cdotp (1\cdotp a)=(a\cdotp 1)\cdotp (a\cdotp a)=b\cdotp a.$$
And finally using the previous equality we obtain
$$(2\cdotp a)=(b\cdotp a)\cdotp (b\cdotp 2)=(a\cdotp b)\cdotp (a\cdotp 2)=1\cdotp 2=b $$
$$(1\cdotp b)=(a\cdotp b)\cdotp (a\cdotp 1)=(b\cdotp a)\cdotp (b\cdotp 1)=2\cdotp 1=a. $$
Therefore the multiplication table is 
\[
\begin{array}{c|cccc}
 & 1 & 2 & a & b \\
\hline
1 & 2 & b & 1 & a \\
2 & a & 1 & b & 2 \\
a & b & 2 & a & 1 \\
b & 1 & a & 2 & b \\
\end{array}
\]
and hence $\{1,2,a,b\}$ is a sub-cycle set of $X$.\\
\textcolor{black}Let $x\in X-\{1,2,a,b\}$, clearly $x\in\Fix T$, by \eqref{eq:M}, $M(x,1)\in (\Fix T\times (X-\Fix T))\cup ((X-\Fix T)\times \Fix T)$. If $(1\cdotp x,x\cdotp 1)\in (X-\Fix T)\times \Fix T$, we must have $1\cdotp x\in \{1,2\}$, hence $x\in \{1,a\}$, a contradiction. Therefore we must have $(1\cdotp x,x\cdotp 1)\in \Fix T\times (X-\Fix T)$, and hence $x\cdotp 1\in \{1,2\}$, and since $\sigma_x$ fixes the blocks, we have $x\cdotp 1=1$. In the same way, one has $x\cdotp 2=2$. Finally, we have
$$b=1\cdotp 2=(x\cdotp 1)\cdotp (x\cdotp 2)=(1\cdotp x)\cdotp (1\cdotp 2)=(1\cdotp x)\cdotp b $$
for all $x\in X-\{1,2,a,b\}$. Since $\{1,2,a,b\}$ is a sub-cycle set, $\sigma_1(X-\{1,2,a,b\})=X-\{1,2,a,b\}$. So we obtain $b=y\cdot b$ for all $y\in X-\{1,2,a,b\}.$
In a similar way, we have $x\cdotp a=a $ for all $x\in \textcolor{black}{X}-\{1,2,a,b\}$. Hence, we have that $\sigma_x(\{1,2,a,b \})\subseteq \{1,2,a,b \}$ for all $x\in X$, hence $X$ is decomposable, a contradiction. Therefore, $|X|=4$. The last part of the statement follows by computer calculation.
\end{proof}

It is natural to ask whether a result analogous to the previous theorem remains valid when the squaring map $T$ is a $p$-cycle, where $p>2$ is a prime. The answer is affirmative under some assumptions, as shown in the following lemma, which will subsequently be used to prove, in a more general setting, the case $p=3$.

\begin{lemma}\label{nonexis}
Let $X$ be a finite indecomposable nilpotent cycle set and
\begin{itemize}
\item $T=(1,2,\ldots, p)$ for some prime $p\ge 3$;
\item $\mathcal{G}(X)$ acts on a partition $I=\{\Delta_{1},\ldots, \Delta_{p}\}$ of size $p$ and with $i\in \Delta_{i}$ for all $i\in \{1,...,p\}$;
\item $\sigma_{{x}_{|_I}}=\sigma_{{y}_{|_I}}=\alpha$, where $\alpha$ is a $p$-cycle of $Sym(I)$, for all $x,y\in \{1,...,p\}$.
\end{itemize}
Then $|X|=p$, moreover there is a unique such cycle set.
\end{lemma}
\begin{proof}
By Proposition \ref{pitype}, $\mathcal{G}(X)$ is a $p$-group and $X$ has prime power size, thus $g_I$ is a $p$-cycle or the identity for all $g\in \mathcal{G}(X)$. Now let $x\in \Fix T$,  then the fact $\sigma_x(x)=x$, forces $\sigma_{{x}_{|_I}}$ to be identity.\\
Let $x,y\in \Fix T$, and $M(x,y)=(y\cdot x, x\cdot y)\in \Fix T\times (X-\Fix T)$, then using the cycloid relation
$$\sigma_{x\cdot y}\sigma_x=\sigma_{y\cdot x}\sigma_y,$$
we obtain that
\begin{align*}
    \sigma_{x\cdot y}\sigma_x(\Delta_1) &=\sigma_{x\cdot y}(\Delta_1)\neq \Delta_1\\
    \sigma_{y\cdot x}\sigma_y(\Delta_1) &=\sigma_{y\cdot x}(\Delta_1)=\Delta_1.
\end{align*}
Which is not possible. Similarly, we can show $M(x,y)\notin (X-\Fix T)\times \Fix T$. Hence we get
$$M\big(\Fix T\times \Fix T\big)\subseteq \big((\Fix T\times \Fix T)\cup ((X-\Fix T)\times(X- \Fix T))\big).$$ In the same way, we can show
$$M\big((X-\Fix T)\times (X-\Fix T)\big)\subseteq \big((\Fix T\times \Fix T)\cup ((X-\Fix T)\times(X- \Fix T))\big).$$
Then, we obtain
$$M\big((\Fix T\times \Fix T)\cup (X-\Fix T)\times (X-\Fix T)\big)=\big((\Fix T\times \Fix T)\cup ((X-\Fix T)\times(X- \Fix T))\big).$$
If $M(x,y)\in (X-\Fix T)\times (X-\Fix T)$ for all $x,y\in X-\Fix T\times X-\Fix T$, since $M$ is bijective, we must have $M(x,y)\in \Fix T\times \Fix T$ for all $x,y\in \Fix T$. This implies that $\Fix T$ and $X-\Fix T$ are sub-cycle set of $X$ and hence $X$ is decomposable, a contradiction. Therefore, there exist $i,j\in X-\Fix T $ and $a,b\in \Fix T\times \Fix T$ such that $M(i,j)=(j\cdotp i,i\cdotp j)=(a,b)$. \\
Now
\[
a=a\cdot a=(j\cdot i)\cdot (j\cdot i)=(i\cdot j)\cdot (i\cdot i)
=b\cdot (i\cdotp i)\in \alpha(\Delta_{i}),
\]
and using this equality again,
\[
a=a\cdot a=(b\cdot (i\cdotp i))\cdot (b\cdot (i\cdotp i))
=((i\cdotp i)\cdot b)\cdot ((i\cdotp i)\cdot (i\cdotp i))
\in \alpha^2(\Delta_{i})\cup \alpha^3(\Delta_{i}),
\]
We obtain
\[
a\in \alpha(\Delta_{i})\cap
(\alpha^2(\Delta_{i})\cup\alpha^3(\Delta_{i})),
\]
but since $p\ge 3$ we have $\alpha(\Delta_{i})\cap
(\alpha^2(\Delta_{i})\cup\alpha^3(\Delta_{i}))=\emptyset $, contradicting the existence of such an element $a$. Thus $\Fix T=\emptyset$, that is, $|X|=p$. The uniqueness of $X$ follows by \cite[Theorem 2.13]{ESS99}.
\end{proof}

\begin{thm}\label{classc3}
    Let $X$ be an indecomposable cycle set with nilpotent permutation group and with $T=(1\quad2\quad 3)$. Then, $|X|=3$. Moreover, there is a unique such a cycle set.
\end{thm}
\begin{proof}
    By Proposition \ref{pitype}, $|X|=3^s$ for some $s\in \mathbb{N}$ and $\mathcal{G}(X)$ is a $3$-group. Suppose $|X|>3$. Then, by Lemma \ref{sysblocks} we have a system of blocks of size $3$. If we indicate by $\Delta_i$ the block of the element $i\in X$, we have to distinguish some cases.
    \begin{center}
        Case 1: $\Delta_1=\Delta_2=\Delta_3$
    \end{center}
    This case is not possible since otherwise $\Delta_1$ is fixed by $\mathcal{G}(X)$ and this implies that $X$ is decomposable.
     \begin{center}
        Case 2: $\Delta_1=\Delta_2$ and $\Delta_2\neq \Delta_3$
    \end{center}   
Since $T=(1\quad 2\quad 3)$, we have that $\sigma_1$ acts as the identity on the blocks, $\sigma_2$ acts on the blocks as the $3$-cycle $(\Delta_1\quad \Delta_3 \quad \Delta)$ (where $\Delta$ it the third block) and $\sigma_3$ acts on the blocks as the $3$-cycle $(\Delta_3\quad \Delta_1 \quad \Delta)$.\\
Now, we have $\sigma_{1\cdotp 2}\sigma_1=\sigma_{2\cdotp 1}\sigma_2$. Since these two permutations must act in the same way on the blocks, and moreover $1\cdotp 2\in \Delta_1$ with $1\cdotp 2\neq 2$, we have necessarily that $\sigma_{1\cdotp 2}$ acts as the identity on the blocks and  $\sigma_{2\cdotp 1}$ acts as $(\Delta_3 \quad \Delta_1\quad \Delta)$. This implies that $2\cdotp 1=3$, a contradiction.
 \begin{center}
        Case 3: $\Delta_1\neq \Delta_2$ and $\Delta_2= \Delta_3$\\
    \end{center}  
It is similar to the Case 2.
\begin{center}
    Case 4: $\Delta_1= \Delta_3$ and $\Delta_2\neq \Delta_3$
\end{center}
It is similar to the Case 2.
\begin{center}
        Case 5: $\Delta_1\neq \Delta_2$, $\Delta_2\neq \Delta_3$ and $\Delta_3\neq \Delta_1$
    \end{center} 
In this case, $\sigma_1$, $\sigma_2$ and $\sigma_3$ act as the $3$-cycle $(\Delta_1\quad \Delta_2\quad \Delta_3)$ on the blocks. Therefore we have $|X|=3$ by Lemma \ref{nonexis} and this is the unique possible case. Unicity follows by \cite[Theorem $2.13$]{ESS99}.
\end{proof}

Under the assumption that $X$ has a $p$-power cardinality, we can obtain the results of the previous theorems, when $p\in \{2,3\}$, without assuming  the nilpotency of $\mathcal{G}(X)$, as we see in the following corollaries.

\begin{cor}\label{classif}
     Let $(X,\cdotp)$ be an indecomposable cycle set of order $2^k$ and with squaring map $T:=(1\quad 2)$. Then $|X|\in \{2,4\}$. Moreover, for each size, the cycle sets are unique.
\end{cor}

\begin{proof}
     By Proposition \ref{cablpi}, there exists a cabling $(X,\cdotp_l)$ of $(X,\cdotp)$ that is of $\pi$-type, hence $\mathcal{G}(X)$ is a $2$-group, and again by Proposition \ref{cablpi} the squaring map of $(X,\cdotp_l)$ is a transposition. By Theorem \ref{classc2} applied to $(X,\cdotp_l)$, $|X|\in \{2,4\}$. The second part follows by \cite[Theorem 2.13]{ESS99} and the inspection of the cycle sets of size $4$.
\end{proof}

\begin{cor}\label{classif2}
     Let $(X,\cdotp)$ be an indecomposable cycle set of order $3^k$ and with squaring map $T:=(1\quad 2\quad 3)$. Then $|X|=3$. Moreover, the cycle sets is unique.
\end{cor}

\begin{proof}
     By Proposition \ref{cablpi}, there exists a cabling $(X,\cdotp_l)$ of $(X,\cdotp)$ that is of $\pi$-type, hence $\mathcal{G}(X)$ is a $3$-group, and again by Proposition \ref{cablpi} the squaring map of $(X,\cdotp_l)$ is a $3$-cycle. By Theorem \ref{classc3} applied to $(X,\cdotp_l)$, $|X|=3$. The second part follows by \cite[Theorem 2.13]{ESS99}.
\end{proof}

Theorems \ref{classc2} and \ref{classc3} answer in negative to Questions 3.18 and 3.19 of \cite{VenRami22}, under the assumption that $\mathcal{G}(X)$ is nilpotent. At present, we are unable to determine whether the results established for $p \in \{2,3\}$ remain valid for an arbitrary prime number without additional assumptions on the blocks. We therefore pose the following question, which close the section.

\begin{qns}
Let $X$ be an indecomposable cycle set with nilpotent permutation group, and suppose that the squaring map $T$ is a $p$-cycle for some prime number $p$ with $p>3$. Is it true that $|X|=p$?\end{qns}

\section{The latin case}

In this section, we study latin cycle sets and we show that, unlike in the general case, the situation is more rigid. In particular, we prove that there is only one latin cycle set with a $p$-cycle as squaring map (without restrictions on the permutation group and on the prime number $p$).

\begin{defn}
    Let $(X,\cdotp)$ be a cycle set. Then, $X$ is said to be \emph{latin} if $X$ is also a quasigroup.
\end{defn}

\noindent Of course, latin cycle sets are always indecomposable and irretractable.\\
In \cite{VenRami22}, Vendramin and Ramirez noted that the squaring map $T$ of indecomposable cycle sets of small size has few fixed points. In particular, computer calculations suggest that  if the squaring map of a cycle set $X$ fixes more than $\frac{|X|}{2}$ points, then $X$ is decomposable. In the following lemma we confirm this evidences for the class of latin cycle sets.

\begin{lemma}\label{lemmlat}
        Let $X$ be a latin cycle set with diagonal map $T$. Then, $|Fix(T)|< \frac{|X|}{2} +1$.
\end{lemma}
\begin{proof} 
   If $|Fix(T)|\geq \frac{|X|}{2} +1$ then $|X|\geq 2(|X|-|Fix(T)|)+2$ and this implies that if $x\in Fix(T)$, there exists $z\in Fix(T)$ such that $z\neq x$ and $x\cdotp z\in Fix(T)$. Then $(x\cdotp z)\cdotp (x\cdotp z)=x\cdotp z$, and $(x\cdotp z)\cdotp (x\cdotp z)=(z\cdotp x)\cdotp (z\cdotp z)=(z\cdotp x)\cdotp z$, therefore $x\cdotp z=(z\cdotp x)\cdotp z$. By the latinity of $X$ we have $x=z\cdotp x$ and since $x\cdotp x=x$, again by the latinity of $X$ we obtain $x=z$, a contradiction.
    \end{proof}
    
\noindent Surprisingly, there is only one latin cycle set whose $T$ is a $p$-cycle. We show this in the following theorem. 

\begin{thm}
    Let $X$ be a latin cycle set and suppose that $T=(1\ldots p)$, where $p$ is a prime number. Then, $p=2$ and $|X|=4$. Moreover, $X$ is unique up to isomorphisms.
\end{thm}
\begin{proof}
    By \cite[Theorem A]{CS23}, $X$ is a multiple of $p$, hence there exist a natural number $k$ such that $|X|=kp$. By the previous proposition, we have the inequality $|X|<2(|X|-|Fix(T)|)+2$ and hence $kp<2p+2$. If $p>2$, we must have $k\in \{1,2\}$, hence $X$, being latin, would be an indecomposable and irretractable cycle set with $|X|$ a square-free number, and this contradicts the main theorem of \cite{CO23}. Therefore, we must have $p=2$ and $|X|\in \{2,4\}$. By inspection of the small cycle sets, the only possibility is $|X|=4$, and there is only one such cycle set, up to isomorphisms.
\end{proof}

\begin{cor}
        Let $X$ be a cycle set such that the diagonal map fixes more than $\frac{|X|}{2}+1$ points. Then, $X$ is not latin. 
\end{cor}
\begin{proof}
    It follows from Lemma \ref{lemmlat}.
\end{proof}

\section*{Competing Interests and Funding}

The authors have not disclosed any competing interests.

\bibliographystyle{plainurl}
\bibliography{reference.bib} 

@article {smock,
    AUTHOR = {Smoktunowicz, Agata and Smoktunowicz, Alicja},
     TITLE = {Set-theoretic solutions of the {Y}ang-{B}axter equation and
              new classes of {R}-matrices},
   JOURNAL = {Linear Algebra Appl.},
  FJOURNAL = {Linear Algebra and its Applications},
    VOLUME = {546},
      YEAR = {2018},
     PAGES = {86--114},
      ISSN = {0024-3795},
   MRCLASS = {15B99 (15A24 15A69 15B10 16N20 16N40 16T25 16T99)},
  MRNUMBER = {3771874},
MRREVIEWER = {Tan Zhang},
       DOI = {10.1016/j.laa.2018.02.001},
       URL = {https://doi.org/10.1016/j.laa.2018.02.001},
}

@article{BA15FA,
  title = {A family of irretractable square-free solutions of the {{Yang-Baxter}} equation},
  author = {Bachiller, D. and Cedó, F. and Jespers, E. and Okniński, J.},
  journal = {Forum Math.},
  volume = {29},
  number = {6},
  pages = {1291--1306},
  url = {https://doi.org/10.1515/forum-2015-0240},
  year = {2017}
}

@article{cacsp2018,
  title = {Indecomposable involutive set-theoretic solutions of the {{Yang-Baxter}} equation},
  author = {Castelli, M. and Catino, F. and Pinto, G.},
  journal = {J. Pure Appl. Algebra},
  volume = {220},
  number = {10},
  pages = {4477--4493},
  url = {https://doi.org/10.1016/j.jpaa.2019.01.017},
  year = {2019}
}

@article{CommAlg,
  author = {Kanrar, A. and Rump, W.},
  title = {A note on braces and {Frobenius} action},
  journal = {Comm. Algebra},
  year = {2025},
  pages = {1--6},
  doi = {10.1080/00927872.2025.2451103}
}

@article{GI-V,
  author = {Gateva-Ivanova, T. and Van den Bergh, M.},
  title = {Semigroups of {I}-type},
  journal = {J. Algebra},
  volume = {206},
  year = {1998},
  pages = {97--112},
  doi = {10.1006/jabr.1998.7444}
}

@article{Baxter,
  author = {Baxter, R. J.},
  title = {Partition function of the eight-vertex lattice model},
  journal = {Ann. Phys.},
  volume = {70},
  year = {1972},
  pages = {193--228},
  doi = {10.1016/0003-4916(72)90335-1}
}

@article{Yang,
  author = {Yang, C. N.},
  title = {Some Exact Results for the Many-Body Problem in One Dimension with Repulsive Delta-Function Interaction},
  journal = {Phys. Rev. Lett.},
  volume = {19},
  year = {1967},
  pages = {1312--1315},
  doi = {10.1103/PhysRevLett.19.1312}
}

@article{CO23,
  author = {Cedó, F. and Okniński, J.},
  title = {Indecomposable solutions of the {Yang-Baxter} equation of square-free cardinality},
  journal = {Adv. Math.},
  volume = {430},
  year = {2023},
  pages = {109221},
  doi = {10.1016/j.aim.2023.109221}
}

@article{CO14,
  title = {Braces and the {{Yang-Baxter}} equation},
  author = {Cedó, F. and Jespers, E. and Okniński, J.},
  journal = {Comm. Math. Phys.},
  volume = {327},
  number = {1},
  pages = {101--116},
  url = {https://doi.org/10.1007/s00220-014-1935-y},
  year = {2014}
}

@article{CA23IND,
  title={On the indecomposable involutive solutions of the {{Yang-Baxter}} equation of finite primitive level},
  author={Castelli, M.},
  journal={Publ. Mat. 69(2)},
  year={2025},
pages={429--444},
url={10.5565/PUBLMAT6922509},
doi={10.5565/PUBLMAT6922509}
}

@article{DI25END,
  title={Endocabling of involutive solutions to the {Yang-Baxter} equation, with an application to solutions whose diagonal is a cyclic permutation},
  author={Dietzel, C.},
  journal={arXiv preprint arXiv:2504.14339},
  year={2025}
}

@article{LRV24,
  author = {Lebed, V. and Remírez, S. and Vendramin, L.},
  title = {Involutive {Yang-Baxter}: cabling, decomposability, Dehornoy class},
  journal = {Rev. Mat. Iberoam.},
  volume = {40},
  number = {2},
  year = {2024},
  pages = {623--635},
  doi = {10.4171/RMI/1438}
}

@article{BOC25,
  title = {Involutive (simple) latin solutions of the {{Yang-Baxter}} equation and related (left) quasigroups},
  author = {Bonatto, M. and Castelli, M.},
  journal = {arXiv preprint arXiv:2501.03660},
  year = {2025}
}

@article{colazzo2024cabling,
  title={On the cabling of non-involutive set-theoretic solutions of the {{Yang--Baxter}} equation},
  author={Colazzo, I. and Van Antwerpen, A.},
  journal={arXiv preprint arXiv:2410.23821},
  year={2024}
}

@article{CS23,
  author = {Camp-Mora, S. and Sastriques, R.},
  title = {A criterion for decomposabilty in {QYBE}},
  journal = {Int. Math. Res. Not.},
  year = {2023},
  number = {5},
  pages = {3808--3813},
  doi = {10.1093/imrn/rnab357}
}

@article{fe24de,
  title = {Dehornoy’s class and Sylows for set-theoretical solutions of the {Yang--Baxter} equation},
  author = {Feingesicht, E.},
  journal = {Int. J. Algebra Comput.},
  volume = {34},
  number = {1},
  pages = {147--173},
  year = {2024},
url={https://doi.org/10.1142/S0218196724500048},
doi={https://doi.org/10.1142/S0218196724500048},
  publisher = {World Scientific}
}

@article{CaCaSt20x,
      AUTHOR = {Castelli, M. and Catino, F. and Stefanelli, P.},
     TITLE = {Indecomposable involutive set-theoretic solutions of the
              {Y}ang-{B}axter equation and orthogonal dynamical extensions
              of cycle sets},
   JOURNAL = {Mediterr. J. Math.},
  FJOURNAL = {Mediterranean Journal of Mathematics},
    VOLUME = {18},
      YEAR = {2021},
    NUMBER = {6},
     PAGES = {Paper No. 246, 27},
      ISSN = {1660-5446},
   MRCLASS = {16T25 (20E22 20N02 81R50)},
  MRNUMBER = {4330443},
MRREVIEWER = {Mafoya Landry Dassoundo},
       DOI = {10.1007/s00009-021-01912-4},
       URL = {https://doi.org/10.1007/s00009-021-01912-4},
}

@article{JP,
  author = {Jedlička, P. and Pilitowska, A.},
  title = {Indecomposable involutive solutions of the {Yang-Baxter} equation of multipermutation level 2 with non-abelian permutation group},
  journal = {J. Combin. Theory Ser. A},
  volume = {197},
  year = {2023},
  pages = {105753},
  doi = {10.1016/j.jcta.2023.105753}
}

@article{jedlivcka2021cocyclic,
  AUTHOR = {Jedli\v{c}ka, P. and Pilitowska, A. and Zamojska-Dzienio,
              A.},
     TITLE = {Cocyclic braces and indecomposable cocyclic solutions of the
              {Y}ang-{B}axter equation},
   JOURNAL = {Proc. Amer. Math. Soc.},
  FJOURNAL = {Proceedings of the American Mathematical Society},
    VOLUME = {150},
      YEAR = {2022},
    NUMBER = {10},
     PAGES = {4223--4239},
      ISSN = {0002-9939},
   MRCLASS = {16T25 (20B35 20D15)},
  MRNUMBER = {4470170},
       DOI = {10.1090/proc/15962},
       URL = {https://doi.org/10.1090/proc/15962},
}

@article{BKSV,
  author = {Bonatto, M. and Kinyon, M. and Stanovský, D. and Vojtěchovský, P.},
  title = {Involutive latin solutions of the {{Yang-Baxter}} equation},
  journal = {J. Algebra},
  volume = {565},
  year = {2021},
  pages = {128--159},
  doi = {10.1016/j.jalgebra.2020.08.032}
}

@article{CJO,
  author = {Cedó, F. and Jespers, E. and Okniński, J.},
  title = {Primitive set-theoretic solutions of the {Yang-Baxter} equation},
  journal = {Comm. Contemp. Math.},
  volume = {24},
  number = {9},
  year = {2022},
  pages = {2150105},
  doi = {10.1142/S021913202150105X}
}

@article{ArpanArch,
  author = {Kanrar, A.},
  title = {{(In)}decomposability of finite solutions of the {Yang-Baxter} equation},
  journal = {Arch. Math. (Basel)},
  volume = {122},
  number = {2},
  year = {2024},
  pages = {155--161},
  doi = {10.1007/s00013-023-01930-6}
}

@article{Okcycle,
  author = {Rump, W.},
  title = {Primes in coverings of indecomposable involutive set-theoretic solutions to the {Yang-Baxter} equation},
  journal = {Bull. Belg. Math. Soc. Simon Stevin},
  volume = {30},
  number = {2},
  year = {2023},
  pages = {260--280},
  doi = {10.5269/bms.v30i2.7349}
}

@article{RumpForum,
  author = {Rump, W.},
  title = {Classification of indecomposable involutive set-theoretic solutions to the {Yang-Baxter} equation},
  journal = {Forum Math.},
  volume = {32},
  number = {4},
  year = {2020},
  pages = {891--903},
  doi = {10.1017/fms.2020.14}
}

@article{Rump05,
  author = {Rump, W.},
  title = {A decomposition theorem for square-free unitary solutions of the quantum {Yang-Baxter} equation},
  journal = {Adv. Math.},
  volume = {193},
  year = {2005},
  pages = {40--55},
  doi = {10.1016/j.aim.2004.05.006}
}

@article{COK,
  title={Constructing finite simple solutions of the {Y}ang-{B}axter equation},
  author={Ced{\'o}, Ferran and Okni{\'n}ski, Jan},
  journal={Adv. Math.},
  volume={391},
  pages={107968},
  url={ http://dx.doi.org/10.1016/j.aim.2021.107968},
  year={2021},
  publisher={Elsevier}
}

@article{VenRami22,
  author = {Ramírez, S. and Vendramin, L.},
  title = {Decomposition theorems for involutive solutions to the {Y}ang-{B}axter equation},
  journal = {Int. Math. Res. Not.},
  year = {2022},
  number = {22},
  pages = {18078--18091},
  doi = {10.1093/imrn/rnab232}
}

@incollection{Drinfeld92,
  author = {Drinfeld, V. G.},
  title = {On some unsolved problems in quantum group theory},
  booktitle = {Quantum groups ({L}eningrad, 1990)},
  series = {Lecture Notes in Math.},
  volume = {1510},
  publisher = {Springer, Berlin},
  year = {1992},
  doi = {10.1007/BFb0101175}
}

@article{ESS99,
  author = {Etingof, P. and Schedler, T. and Soloviev, A.},
  title = {Set-theoretical solutions to the quantum {Y}ang-{B}axter equation},
  journal = {Duke Math. J.},
  volume = {100},
  year = {1999},
  number = {2},
  pages = {169--209},
  doi = {10.1215/S0012-7094-99-10007-X}
}

@article{Rump07,
  author = {Rump, W.},
  title = {Braces, radical rings, and the quantum {Y}ang-{B}axter equation},
  journal = {J. Algebra},
  volume = {307},
  year = {2007},
  number = {1},
  pages = {153--170},
  doi = {10.1016/j.jalgebra.2006.03.040}
}

@article{smoktunowicz18,
  author = {Smoktunowicz, A.},
  title = {On {E}ngel groups, nilpotent groups, rings, braces and the {Y}ang-{B}axter equation},
  journal = {Trans. Amer. Math. Soc.},
  volume = {370},
  year = {2018},
  number = {9},
  pages = {6535--6564},
  doi = {10.1090/tran/7179}
}

\end{document}